\documentclass[10pt]{article}

\usepackage{amssymb,amsthm,amsmath,hyperref}
\numberwithin{equation}{section}

\newtheorem{thm}{Theorem}[section]
\newtheorem{lem}{Lemma}[section]

\theoremstyle{definition}

\theoremstyle{remark}
\newtheorem{rem}{Remark}[section]

\allowdisplaybreaks

\begin{document}

\title{Global Small Solutions to a Complex Fluid Model in 3D}

\author{Fanghua Lin$^1$, Ting Zhang$^2$\\
1. Courant Institute, New York University, New York, NY 10012, USA\\
2. Department of Mathematics, Zhejiang University,
Hangzhou 310027, China}
\date{}
\maketitle

\begin{abstract}
     In this paper, we provide a much simplified proof of the main result
in \cite{Lin12} concerning the global existence and uniqueness of smooth
solutions to the Cauchy problem for a 3D incompressible complex fluid
model under the assumption that the initial data are close to some
equilibrium states. Beside the classical energy method, the interpolating
inequalities and the algebraic structure of the equations coming from the
incompressibility of the fluid are crucial in our arguments. We
combine the energy estimates with the $L^\infty$ estimates for time
slices to deduce the key $L^1$ in time estimates. The latter is
responsible for the global in time existence.
\end{abstract}

\section{Introduction.}
In this paper, we consider the global existence of classical solutions to the
following simple model for a complex fluid flows
\begin{equation}
  \left\{
  \begin{array}{l}
    \partial_t \phi +v\cdot\nabla \phi=0,\ \ (t,x)\in\mathbb{R}^+\times\mathbb{R}^3,\\
        \partial_t v+v\cdot\nabla v-\Delta v+\nabla p=-\mathrm{div}[\nabla \phi\otimes\nabla \phi],\\
            \mathrm{div}v=0,\\
            (\phi,v)|_{t=0}=(\phi_0,v_0).
  \end{array}
  \right.\label{m0-E1.1-1111}
\end{equation}
 Here $\phi$,  $v=(v_1,v_2,v_3)^\top$ and $p$
denote the scalar potential, velocity field and scalar pressure of the fluid respectively.
As in \cite{Lin12}, we shall consider the initial data $(\phi_0,v_0)$ is close to a non-trivial equilibrium, e.g. $(\phi_0,v_0)\simeq(x_3,(0,0,0)^\top)$. Note that instead of $x_3$, any non-constant linear functions would work as
well by our method (cf.\cite{Lin12} for explanations). What would be important is that $\nabla \phi_0$ is close to a constant non-zero vector field in a suitable way. Thus, we may write $(\phi,v)=(x_3+\psi,v)$, and
substitute it into (\ref{m0-E1.1-1111}) with $x\in \mathbb{R}^3$, to obtain the following equivalent system for $(\psi,v)$,
 \begin{equation}
  \left\{
  \begin{array}{l}
    \partial_t \psi +v\cdot\nabla \psi+  v_3=0,\ \ (t,x)\in\mathbb{R}^+\times\mathbb{R}^3,\\
        \partial_t v_h+v\cdot\nabla v_h -\Delta v_h+\nabla_h p+ \nabla_h\partial_3\psi=
        -\mathrm{div}[\nabla_h \psi\otimes\nabla \psi],\\
        \partial_t v_3+v\cdot\nabla v_3 -\Delta v_3+\partial_3 p+ (\Delta+\partial_3^2)\psi=
        -\mathrm{div}[\partial_3 \psi \nabla \psi],\\
            \mathrm{div}v=0,\\
            (\psi,v)|_{t=0}=(\psi_0,v_0),
  \end{array}
  \right.\label{m0-E1.5-N3}
\end{equation}
where $v_h=(v_1,v_2)^\top$, $\nabla_h=(\partial_1,\partial_2)^\top$. For the rest of the paper we shall work on the equations (\ref{m0-E1.5-N3}).

We remark that the nonlinear hyperbolic-parabolic system (\ref{m0-E1.1-1111}) has been used for describing many fluid dynamic models, see \cite{Constantin,Constantin2,Lin12-3}.
Indeed, when $\phi=(\phi_1,\phi_2)^\top$ is a vector-valued function on $\mathbb{R}^+\times\mathbb{R}^2$ with $\det(\nabla\phi)=1$, the system (\ref{m0-E1.1-1111}) is equivalent to the well-known Oldroyd-B model for viscoelastic
fluids equations, see \cite{Saut,Larson,Lin12-2,Renardy}. It is also closely related to the
evolution equation of nematic liquid crystal as well as the diffusive sharp
interface motion and immersed boundary in flow fields \cite{Peskin}, see
also the recent survey article \cite{Lin12-3}. In \cite{Lin12}, authors
used the system (\ref{m0-E1.1-1111}) as a toy model for the 3D incompressible
viscous and non-resistive MHD system.

In fact, the system (\ref{m0-E1.1-1111}) is exactly the incompressible MHD equations with zero magnetic diffusion when the space dimension is two. Recall that the 2D incompressible MHD system reads,
    \begin{equation}
      \left\{
      \begin{array}{l}
        \partial_t b+v\cdot\nabla b-\eta\Delta b=b\cdot\nabla v,\ \ (t,x)\in\mathbb{R}^+\times\mathbb{R}^2,\\
        \partial_t v+v\cdot\nabla v-\nu \Delta v+\nabla p=b\cdot\nabla b,\\
            \mathrm{div}v=\mathrm{div}b=0,\\
                b|_{t=0}=b_0,\ v|_{t=0}=v_0,
      \end{array}
      \right.\label{m0-E1.2}
    \end{equation}
where $b=(b_1,b_2)^{\top}$,   $v=(v_1,v_2)^{\top}$ and $p$
denote the magnetic field, velocity field and scalar pressure of the fluid respectively.
 In (\ref{m0-E1.2}), the condition $\mathrm{div}b=0$ implies the existence of a scalar function $\phi$ such that $b=(\partial_2\phi,-\partial_1\phi)^\top$, and the corresponding system becomes  the following 2D incompressible MHD type system,
\begin{equation}
  \left\{
  \begin{array}{l}
    \partial_t \phi +v\cdot\nabla \phi-\eta\Delta\phi=0,\ \ (t,x)\in\mathbb{R}^+\times\mathbb{R}^2,\\
        \partial_t v+v\cdot\nabla v-\Delta v+\nabla p=-\mathrm{div}[\nabla \phi\otimes\nabla \phi],\\
            \mathrm{div}v=0,\\
            (\phi,v)|_{t=0}=(\phi_0,v_0).
  \end{array}
  \right.
\end{equation}
There are some global wellposedness results for the system (\ref{m0-E1.2}),
see \cite{Duvaut,Sermange} for the case when $\eta>0$ and $\nu>0$, \cite{Cao}
for the case when $\eta>0$ and $\nu=0$, as well as the case with mixed
partial dissipation and additional (artificial) magnetic diffusion.
In \cite{Bardos}, authored considered the case when $\eta=\nu=0$
and the initial data $(b_0,v_0)$ close to the equilibrium state $(B_0,0)$.
In \cite{Lin12-2}, the case when $\eta=0$, $\nu>0$ was studied. Under
the assumption that the initial data
$(b_0,v_0)$ is close to the equilibrium state $((1,0)^\top,0)$, the
global wellposedness was proven. We should note that authors observed
in \cite{Bardos} that the fluctuations $v+b-B_0$ and $v-b+B_0$ propagate along
the $B_0$ magnetic field in opposite directions. Thus, a strong enough
magnetic field will reduce the nonlinear interactions and prevent
formations of strong gradients \cite{Bardos,Frisch,Kraichnan}.
Unfortunately, the method applied in \cite{Bardos} is purely hyperbolic
(characteristic method) and hence could not be applied in our case.

In \cite{Lin12}, using the anisotropic Littlewood-Paley analysis,  the first
author and Ping Zhang \cite{Lin12} proved a global wellposedness result of
the  system (\ref{m0-E1.1-1111}). The arguments involved, despite its
general interests, were rather complicated. The aim of this note is to
give a new and simple proof, which involves only the energy estimate
method, interpolating inequalities and couple elementary observations.

\begin{thm}\label{m0-Thm1.2}
Assume that the initial data $(\psi_0,v_0)$ satisfy
 $
     (\nabla \psi_0,v_0)\in
        H^2(\mathbb{R}^3)\times H^2(\mathbb{R}^3),
   $ $\mathrm{div}v_0=0$,
  then there exists
    a positive constant $c_0$ such that if
        \begin{equation}
         B_0=\|\nabla \psi_0\|_{H^2}+
                  \|v_0\|_{H^2}\leq c_0,\label{1.6}
        \end{equation}
    then the system (\ref{m0-E1.5-N3}) has a unique global solution $(\psi,v,\nabla p)\in F^2$ satisfying
                 \begin{eqnarray}
      B_T^2=\|v\|^2_{L^\infty([0,T]; H^2 )}+\|\nabla \psi\|^2_{L^\infty([0,T];H^2)}
      +\|\nabla v\|^2_{L^2([0,T]; H^2 )}+\|\nabla_h\nabla \psi\|^2_{L^2([0,T]; H^1 )}
      \leq CB_{0}^2,\label{1.7}
    \end{eqnarray}
and
    \begin{equation}
      \|\nabla p\|_{L^\infty([0,T]; H^1 )}
       \leq  C B_0,\label{1.8}
    \end{equation}
    for all $T>0$,
where $C$ is a positive constant independent of $T$,
$$
 F^n= \left\{ (\psi,v,\nabla p)\left| \begin{array}{l}
   (\nabla \psi,v,\nabla p)\in C([0,\infty);H^n\times H^n)\times C([0,\infty); H^{n- 1}), \\
     (\nabla_h\nabla \psi, \nabla v)
   \in L^2([0,\infty);H^{n- 1}\times H^n).
   \end{array}
   \right.
   \right\}
    $$
\end{thm}

We note that under the assumptions of Theorem \ref{m0-Thm1.2}, if $(\nabla
\psi_0,v_0)\in H^n(\mathbb{R}^3)\times H^n(\mathbb{R}^3)$, $n\geq3$, then we can easily obtain that $(\psi,v,\nabla p)\in F^n$
    and omit the details.

There are three key technical points in our proofs:
\begin{description}
  \item[(1)] interpolating estimates, see for example, Lemma \ref{m0-L2.2};
  \item[(2)] using the algebraic structure: $\mathrm{div} v =0$ to
inter-changing the estimates for the vertical ($\partial_3$)  and the
horizontal ($\nabla_h$)  derivatives;
  \item[(3)]  using the first equation of (\ref{m0-E1.5-N3}) to reduce
"$L^1$ in time estimates" (\cite{Lin12}) which is the key to the global
existence result to "energy estimates and
$L^\infty$ estimates for time slices" that are relatively easy to obtain.
\end{description}
 In fact, the basic strategy for the proofs is rather clear. Using the
basic energy laws, one reduces the problems to estimating certain
terms of particular forms.
For example,   one of the difficulties of the proofs would be to control
the following type term,
    \begin{equation}
      \int^T_0\int_{\mathbb{R}^3} \partial_3 v_3  (\partial_3^3\psi)^2 dxdt.
    \end{equation}
Since the horizontal derivatives of $\psi$, $\nabla_h\psi$ decay faster
than $\partial_3\psi$ (by energy laws), in \cite{Lin12} authors explored
such anisotropic behavior by using the
anisotropic Littlewood-Paley theory to conclude the key estimate that
$v_3\in L^1(\mathbb{R}^+;Lip(\mathbb{R}^3))$.
Here we will show that first by interpolating inequalities $\nabla
\psi\in L^4_T(L^{ \infty})$ in Lemma \ref{m0-L2.2}.
Then we use the first equation of  (\ref{m0-E1.5-N3})$_1$ twice and
proceed the estimates as follows:
    \begin{eqnarray}
     &&\left| \int^T_0\int_{\mathbb{R}^3} \partial_3 v_3 (\partial_3^3\psi)^2 dxdt\right|\nonumber\\
        &=&\left|\int^T_0\int_{\mathbb{R}^3}
         \partial_3(\partial_t\psi+v\cdot\nabla\psi  )(\partial_3^3\psi)^2 dxdt\right|\nonumber\\
        &\leq&\left|\int_{\mathbb{R}^3}
         \partial_3 \psi(\partial_3^3\psi)^2 dx\big|^T_0\right|+\left|\int^T_0\int_{\mathbb{R}^3}
        \partial_3 v_3\partial_3\psi(\partial_3^3\psi)^2 dxdt\right|+\ldots\nonumber\\
           &=&\ldots+\left|\int^T_0\int_{\mathbb{R}^3}
        \partial_3 (\partial_t\psi+v\cdot\nabla\psi ) \partial_3\psi(\partial_3^3\psi)^2 dxdt\right|+\ldots\nonumber\\
        &\leq&\ldots+ C\|\nabla
\psi\|_{L^4_T(L^\infty(\mathbb{R}^3))}^2\|\nabla v\|_{L^2_T(L^2(\mathbb{R}^3))}\|\partial_3^3 \psi\|_{L^\infty_T(L^2(\mathbb{R}^3))}^2+\ldots.\label{idea}
    \end{eqnarray}
We refere the details to Lemma \ref{m0-L2.5}.
This is a simple idea works well for the issue concerning various
anisotropic dissipative system similar to (\ref{m0-E1.5-N3}).

We also note that in a recent preprint \cite{Lin12-2} authores embedded
the system
(\ref{m0-E1.2}) with $\eta=0$ into a 2D viscoelastic fluid system.
Then they use equations in Lagrangian coordinates and the anisotropic
Littlewood-Paley analysis techniques,  to obtain a global wellposedness
result. One can apply the methods in this paper to obtain similar
results as theirs.

The organization of this paper can be as the following:  we shall present
some \textit{a priori} estimates in Section \ref{m0-S2}, and prove Theorem
\ref{m0-Thm1.2} in Section
\ref{m0-S3}.

Let us complete this section by the notation we shall use in this paper.

\textbf{Notation.}   We shall denote by $(a|b)$ the $L^2$ inner product of $a$ and $b$, and $(a|b)_{H^s}$ the standard $H^s$ inner product of $a$ and $b$.   Finally, we denote $L^p_T(L^q_h(L^r_v))$ the space
$L^p([0,T];L^q(\mathbb{R}_{x_1}\times\mathbb{R}_{x_{2}};L^r(\mathbb{R}_{x_3})))$, $C_T(X)$ the space $C([0,T];X)$.

\section{\textit{A priori} estimates}\label{m0-S2}
 In this section, we prove a set of \textit{a priori} estimates which
are crucial for the global existence of solutions for the system
(\ref{m0-E1.5-N3}). We begin with the following
Gagliardo-Nirenberg-Sobolev type estimate, see \cite{Nirenberg}.

\begin{lem}\label{m0-L2.2}
  If the function $\psi$ satisfies that $\nabla \psi\in L^\infty_T(H^2(\mathbb{R}^3))$ and $\nabla_h\nabla \psi\in L^2_T(H^1(\mathbb{R}^3))$,  then there hold
  \begin{eqnarray}
      \|\nabla \psi\|_{L^4_T(L^4(\mathbb{R}^3))}   &\leq& C      \| \nabla_h\nabla \psi\|_{L^2_T(L^2(\mathbb{R}^3))}^\frac{1}{2}
        \|   \nabla \psi\|_{L^\infty_T(H^1(\mathbb{R}^3))}^\frac{1}{2},\label{3D-E2.3-00}
    \end{eqnarray}
     \begin{equation}
      \|\nabla^2 \psi\|_{L^4_T(L^4(\mathbb{R}^3))}\leq C      \| \nabla_h\nabla \psi\|_{L^2_T(H^1(\mathbb{R}^3))}^\frac{1}{2}
        \|   \nabla \psi\|_{L^\infty_T(H^2(\mathbb{R}^3))}^\frac{1}{2},\label{3D-E2.3-0}
    \end{equation}
     \begin{equation}
      \|\nabla \psi\|_{L^4_T(L^{\infty}(\mathbb{R}^3))}\leq C      \| \nabla_h\nabla \psi\|_{L^2_T(H^1(\mathbb{R}^3))}^\frac{1}{2}
        \|   \nabla \psi\|_{L^\infty_T(H^2(\mathbb{R}^3))}^\frac{1}{2},\label{m0-E2.2}
    \end{equation}
    where   $C$ is a positive constant  independent of $T$.
\end{lem}
\begin{proof}
Using Gagliardo-Nirenberg-Sobolev's inequality and Minkowski's inequality, we obtain
   \begin{eqnarray}
      \|\nabla \psi\|_{L^4_T(L^4(\mathbb{R}^3))}&\leq& C
      \left\|
      \|\nabla  \psi\|_{L^2_h}^\frac{1}{2}
            \|\nabla  \nabla_h  \psi\|_{L^2_h}^\frac{1}{2}
      \right\|_{L^4_T(L^4_v)}\nonumber\\
      &\leq&C
      \|\nabla  \psi\|_{L^\infty_T(L^\infty_v(L^2_h))}^\frac{1}{2}
            \|\nabla  \nabla_h  \psi\|_{L^2(L^2(\mathbb{R}^3))}^\frac{1}{2}\nonumber\\
      &\leq&C
      \|\nabla  \psi\|_{L^\infty_T(L^2_h(L^\infty_v))}^\frac{1}{2}
            \|\nabla  \nabla_h  \psi\|_{L^2(L^2(\mathbb{R}^3))}^\frac{1}{2}\nonumber\\
      &\leq&C
      \left\|
      \|\nabla  \psi\|_{L^2_v}^\frac{1}{2}
            \|\nabla  \partial_3\psi\|_{L^2_v}^\frac{1}{2}
      \right\|_{L^\infty_T(L^2_h)}^\frac{1}{2}
        \|\nabla  \nabla_h  \psi\|_{L^2(L^2(\mathbb{R}^3))}^\frac{1}{2}\nonumber\\
      &\leq&C
      \|\nabla \psi\|_{L^\infty_T(L^2(\mathbb{R}^3))}^\frac{1}{4}
            \|\nabla  \partial_3\psi\|_{L^\infty_T(L^2(\mathbb{R}^3))}^\frac{1}{4}
             \|\nabla  \nabla_h  \psi\|_{L^2_T(L^2(\mathbb{R}^3))}^\frac{1}{2}\nonumber\\
            &\leq& C      \| \nabla_h\nabla \psi\|_{L^2_T(L^2(\mathbb{R}^3))}^\frac{1}{2}
        \|   \nabla \psi\|_{L^\infty_T(H^1(\mathbb{R}^3))}^\frac{1}{2}.
    \end{eqnarray}
This proves (\ref{3D-E2.3-00}). Inequality (\ref{3D-E2.3-0}) is a direct
consequence of (\ref{3D-E2.3-00}). Combining
(\ref{3D-E2.3-00})-(\ref{3D-E2.3-0}) with
Gagliardo-Nirenberg-Sobolev's inequality again, we obtain (\ref{m0-E2.2}).
\end{proof}

By taking divergence of the $v$ equation of (\ref{m0-E1.5-N3}), we can
express the pressure function $p$ via
    \begin{equation}
      p=-2 \partial_3 \psi +\sum_{i,j=1}^3(-\Delta)^{-1}
      [\partial_iv_j\partial_j v_i+\partial_i\partial_j(\partial_i \psi\partial_j \psi)
      ].\label{m0-E1.6}
    \end{equation}
As in \cite{Lin12}, we substitute (\ref{m0-E1.6}) into (\ref{m0-E1.5-N3})
to obtain
    \begin{equation}
  \left\{
  \begin{array}{l}
    \partial_t \psi +v\cdot\nabla \psi + v_3=0,\ \ (t,x)\in\mathbb{R}^+\times\mathbb{R}^3,\\
        \partial_t v_h+v\cdot\nabla v_h -\Delta v_h - \nabla_h\partial_3\psi =f^h\\
     \ \ \ \    \ \ \ \:=-\displaystyle{\sum_{i,j=1}^3}\nabla_h(-\Delta)^{-1}
      [\partial_iv_j\partial_j v_i+\partial_i\partial_j(\partial_i \psi\partial_j \psi)
      ]-\displaystyle{\sum_{j=1}^3}\partial_j[\nabla_h \psi\partial_j \psi],\\
        \partial_t v_3+v\cdot\nabla v_3 -\Delta v_3 + \Delta_h\psi =f^v\\
    \ \ \ \   \ \ \ \:=-\displaystyle{\sum_{i,j=1}^3}\partial_3(-\Delta)^{-1}
      [\partial_iv_j\partial_j v_i+\partial_i\partial_j(\partial_i \psi\partial_j \psi)
      ]-\displaystyle{\sum_{j=1}^3}\partial_j[\partial_3 \psi\partial_j \psi],\\
            \mathrm{div}v=0,\\
            (\psi,v)|_{t=0}=(\psi_0,v_0).
  \end{array}
  \right.\label{m0-E1.7}
\end{equation}
Here, $\Delta_h=\partial_{x_1}^2+\partial_{x_{2}}^2$.

The next Lemma is a standard energy estimate.
\begin{lem}\label{m0-L2.6}
 Let   $(\psi,v)$ be sufficiently smooth functions which solve (\ref{m0-E1.5-N3}),
  then there holds
    \begin{eqnarray}
   &&\frac{d}{dt}\left\{
   \frac{1}{2}\left(
   \|v\|_{H^2}^2+\|\nabla \psi\|_{H^2}^2+\frac{1}{4 }\|\Delta\psi\|_{H^1}^2
   \right)+\frac{1}{4}(v_3|\Delta\psi)_{H^1}
   \right\}      \label{m0-L2.6-111} \\
        &&+\|\nabla v \|_{H^2}^2-\frac{1 }{4}\|\nabla v_3\|_{H^1}^2
        +\frac{1 }{4}\|\nabla\nabla_h\psi\|_{H^1}^2\nonumber\\
   &=&-(v\cdot\nabla v|v)_{H^2}
   +(v\cdot\nabla\psi|\Delta\psi)_{H^2}
   -(\mathrm{div}(\nabla\psi\otimes\nabla\psi)|v)_{H^2}
   -\frac{1}{4}(v\cdot\nabla v_3|\Delta\psi)_{H^1}
   \nonumber\\
        &&
   +\frac{1}{4}(f^v|\Delta\psi)_{H^1}
   +\frac{1}{4}(\nabla v_3|\nabla (v\cdot\nabla \psi))_{H^1}
   -\frac{1}{4 }(\Delta(v\cdot\nabla\psi)|\Delta\psi)_{H^1} .
   \nonumber
    \end{eqnarray}
\end{lem}
\begin{proof}
Taking the standard $H^2$ inner product of (\ref{m0-E1.5-N3})$_{2,3}$ with
$v$ and then using the integration by parts, we have
    \begin{eqnarray}
      &&\frac{1}{2}\frac{d}{dt}\|v\|_{H^2}^2
      +(v\cdot\nabla v|v)_{H^2}
            +\|\nabla v\|_{H^2}^2\nonumber\\
            &=&-( \nabla_h\partial_3\psi|v_h)_{H^2}
            -( (\Delta+\partial_3^2)\psi|v_3)_{H^2}
            -(\mathrm{div}(\nabla\psi\otimes\nabla \psi)|v)_{H^2}.\label{m0-E2.11}
    \end{eqnarray}
Since $\mathrm{div}v=0$, the integration by parts gives
        \begin{eqnarray}
        & &-( \nabla_h\partial_3\psi|v_h)_{H^2}
            = ( \partial_3 \psi|\mathrm{div}_h v_h)_{H^2}\nonumber\\
    & =& -(  \partial_3 \psi|\partial_3v_3)_{H^2}
          =  (  \partial_3^2 \psi| v_3)_{H^2}.
    \end{eqnarray}
From (\ref{m0-E1.5-N3})$_1$, we have
        \begin{eqnarray}
        & &-( \Delta\psi|v_3)_{H^2}
            = (  \Delta \psi|(\partial_t\psi+v\cdot\nabla \psi ))_{H^2}\nonumber\\
    & =&-\frac{1}{2}\frac{d}{dt}\|\nabla\psi\|_{H^2}^2+
    (  \Delta \psi| v\cdot\nabla \psi )_{H^2} .\label{m0-E2.13}
    \end{eqnarray}
Combining (\ref{m0-E2.11})-(\ref{m0-E2.13}), we deduce that
    \begin{eqnarray}
      &&\frac{1}{2}\frac{d}{dt}\left(\|v\|_{H^2}^2
      +\|\nabla\psi\|_{H^2}^2\right)
      +\|\nabla v\|_{H^2}^2\nonumber\\
            &=& -(v\cdot\nabla v|v)_{H^2}
         +
    (  \Delta \psi| v\cdot\nabla \psi )_{H^2}
            -(\mathrm{div}(\nabla\psi\otimes\nabla \psi)|v)_{H^2}.\label{m0-E2.14}
    \end{eqnarray}
Next we take the standard $H^1$ inner product of (\ref{m0-E1.7})$_{3}$
with $\Delta\psi$,
 using again the integration by parts, to obtain
    \begin{eqnarray}
      (\partial_tv_3|\Delta\psi)_{H^1}+(v\cdot\nabla v_3|\Delta\psi)_{H^1}
      -(\Delta v_3|\Delta\psi)_{H^1}
     &=&  - \|\nabla_h\nabla \psi\|_{H^1}^2
     +(f^v|\Delta\psi)_{H^1} .\label{m0-E2.15}
    \end{eqnarray}
From the equation (\ref{m0-E1.5-N3})$_1$ and again the integration by
parts, we get
    \begin{eqnarray}
      &&(\partial_tv_3|\Delta\psi)_{H^1}\nonumber\\
            &=&\frac{d}{dt}(v_3|\Delta\psi)_{H^1}
            -(v_3|\Delta\partial_t\psi)_{H^1}\nonumber\\
      &=&\frac{d}{dt}(v_3|\Delta\psi)_{H^1}
           +(v_3|\Delta(v\cdot\nabla\psi + v_3))_{H^1}\nonumber\\
      &=&\frac{d}{dt}(v_3|\Delta\psi)_{H^1}
           -(\nabla v_3|\nabla(v\cdot\nabla\psi))_{H^1}
           -\|\nabla v_3\|_{H^1}^2.
    \end{eqnarray}
We observe, by (\ref{m0-E1.5-N3})$_1$, that
    \begin{eqnarray}
      &&-(\Delta v_3|\Delta\psi)_{H^1}\nonumber\\
            &=&(\Delta(\partial_t\psi+v\cdot\nabla\psi )
            |\Delta\psi)_{H^1}\nonumber\\
      &=&\frac{1}{2 }\frac{d}{dt}\|\Delta\psi\|_{H^1}^2
      +(\Delta( v\cdot\nabla\psi )
            |\Delta\psi)_{H^1}.\label{m0-E2.17}
    \end{eqnarray}
Combining (\ref{m0-E2.15})-(\ref{m0-E2.17}), we hence conclude
    \begin{eqnarray}
      &&\frac{d}{dt}\left\{
      \frac{1}{2}\|\Delta\psi\|_{H^1}^2+(v_3|\Delta\psi)_{H^1}
      \right\}+ \|\nabla_h\nabla\psi\|_{H^1}^2- \|\nabla v_3\|_{H^1}^2\nonumber\\
            &=&-(v\cdot\nabla v_3|\Delta\psi)_{H^1}
                        +(f^v|\Delta\psi)_{H^1}+(\nabla v_3|\nabla(v\cdot\nabla \psi))_{H^1}
            -(\Delta(v\cdot\nabla\psi)|\Delta\psi) .\label{m0-E2.18}
    \end{eqnarray}
With (\ref{m0-E2.14}) and (\ref{m0-E2.18}), one can complete the proof.
\end{proof}

The following is the key \textit{a priori} estimate which is essential to the proof
of the main result of this paper.
\begin{lem}\label{m0-L2.3}
Let   $(\psi,v)$ be sufficiently smooth functions which solve (\ref{m0-E1.5-N3}) and  satisfy
  $\nabla \psi\in L^\infty_T(H^2(\mathbb{R}^3))$, $\nabla_h\nabla \psi\in L^2_T(H^1(\mathbb{R}^3))$,
   $v\in L^\infty_T(H^2(\mathbb{R}^3))$ and $\nabla v\in L^2_T(H^2(\mathbb{R}^3))$, then there holds
    \begin{equation}
   B_T^2    \leq C
               (\|v_0\|_{H^2(\mathbb{R}^3)}^2+\|\nabla\psi_0\|_{H^2(\mathbb{R}^3)}^2)
               + C  B_T^3(1+ B_T)^2.\label{m0-E2.19}
    \end{equation}
    where   $C$ is a positive constant independent of $T$.
\end{lem}
\begin{proof}
By the energy estimate (\ref{m0-L2.6-111}) and the definition of $ B_T$,  we get
for a positive constant $C$ (independent of $T$) that
    \begin{eqnarray}
      B_T^2&\leq&CB_0^2+C\left|\int^T_0(v\cdot\nabla v|v)_{H^2}dt\right|
   +C\left|\int^T_0(v\cdot\nabla\psi|\Delta\psi)_{H^2}dt\right|
   \nonumber\\
        &&
   +C\left|\int^T_0(\mathrm{div}(\nabla\psi\otimes\nabla\psi)|v)_{H^2}dt\right| +C\left|\int^T_0(v\cdot\nabla v_3|\Delta\psi)_{H^1}dt\right|
    +C\left|\int^T_0(f^v|\Delta\psi)_{H^1}dt\right|\nonumber\\
        &&
   +C\left|\int^T_0(\nabla v_3|\nabla (v\cdot\nabla \psi))_{H^1}dt\right|
   +C\left|\int^T_0(\Delta(v\cdot\nabla\psi)|\Delta\psi)_{H^1} dt\right|\nonumber\\
   &:=&CB_0^2+\sum_{j=1}^7I_j.\label{m0-E2.18-00}
    \end{eqnarray}

       We are going to estimate term by term the right hand side of the
inequality. The basic strategies involved in estimating all such quantities
are the same. More precisely, we estimate separately terms involving
horizontal derivatives and terms with vertical derivatives. For terms with
horizontal derivatives $\nabla_h \psi$, one can use the dissipations implied by the energy
equality (\ref{m0-L2.6-111}). For terms containing vertical derivatives, we
use the algebriac relation (deduced from that $\mathrm{div}v=0$) and the
transport equations. The latter reduces space-time estimates to bounds
on time-slices and terms with either horizontal derivatives or of higher
order nonlinearities (hence they are smaller under our smallness
assumptions on the initial data).
       To illustrate the basic idea, we start with the second term $I_2$.
Applying the Gagliardo-Nirenberg-Sobolev type estimates in Lemma \ref{m0-L2.2}, and
use the fact that $\mathrm{div}v=0$,
H\"{o}lder and Sobolev inequalities, we deduce that
    \begin{eqnarray}
     && I_2=C\left|\int^T_0(v\cdot\nabla\psi|\Delta\psi)_{H^2}dt\right|\nonumber\\
            &=&C\left| \sum_{|\alpha|\leq2}\sum_{i=1}^3\int^T_0\int
            [\partial^\alpha\partial_i(v\cdot\nabla\psi)-v\cdot\nabla\partial^\alpha\partial_i\psi]\partial^\alpha\partial_i\psi dxdt
            \right|
            \nonumber\\
     &\leq& C\|\nabla v\|_{L^2_T(L^2(\mathbb{R}^3))}\|\nabla \psi\|_{L^4_T(L^4(\mathbb{R}^3))}^2
     +C\|\nabla^2\psi\|_{L^4_T(L^4(\mathbb{R}^3))}(\|\nabla^2 v\|_{L^2_T(L^2(\mathbb{R}^3))}\|\nabla\psi\|_{L^4_T(L^4(\mathbb{R}^3))}\nonumber\\
     &&
     +\|\nabla  v\|_{L^2_T(L^2(\mathbb{R}^3))}\|\nabla^2\psi\|_{L^4_T(L^4(\mathbb{R}^3))})
      +C\|\nabla_h\nabla^2\psi\|_{L^2_T(L^2(\mathbb{R}^3))}(\|\nabla^3 v\|_{L^2_T(L^2(\mathbb{R}^3))}\|\nabla\psi\|_{L^\infty_T(L^\infty(\mathbb{R}^3))}
    \nonumber\\
                &&  +\|\nabla^2 v\|_{L^2_T(L^4(\mathbb{R}^3))}\|\nabla^2\psi\|_{L^\infty_T(L^4(\mathbb{R}^3))}
          +\|\nabla  v\|_{L^2_T(L^\infty(\mathbb{R}^3))}\|\nabla^3 \psi\|_{L^\infty_T(L^2(\mathbb{R}^3))}
     )\nonumber\\
             &&+C\|\nabla^3\psi\|_{L^\infty_T(L^2(\mathbb{R}^3))}
             (\|\nabla^3 v_h\|_{L^2_T(L^2(\mathbb{R}^3))}\|\nabla_h\psi\|_{L^2_T(L^\infty(\mathbb{R}^3))}
            +\|\nabla^2 v_h\|_{L^2_T(L^4(\mathbb{R}^3))}\|\nabla\nabla_h\psi\|_{L^2_T(L^4(\mathbb{R}^3))}\nonumber\\
                &&
          +\|\nabla  v_h\|_{L^2_T(L^\infty(\mathbb{R}^3))}\|\nabla^2\nabla_h \psi\|_{L^2_T(L^2(\mathbb{R}^3))}
     )\nonumber\\
        &&+C\left|\int^T_0\int \partial_3^3\psi(\partial_3^3v_3\partial_3\psi+
        3\partial_3^2v_3\partial_3^2\psi
        +3\partial_3v_3\partial_3^3\psi)dxdt\right|. \label{m0-E2.19-0}
    \end{eqnarray}
To estimate the last term in the above inequality (\ref{m0-E2.19-0}), we
need the following two technical lemmas. The proofs of these two
Lemmas will be given in the Appendix.
\begin{lem}\label{m0-L2.4}
 Under the conditions in Lemma \ref{m0-L2.3}, then there holds
    \begin{eqnarray}
   && \left|\int^T_0\int_{\mathbb{R}^3}  \partial_3\psi  \partial_3^3\psi\partial_3^3 v_3  dxdt
    \right|
    +\left|\int^T_0\int_{\mathbb{R}^3}  \partial_3^2\psi  \partial_3^3\psi\partial_3^2 v_3  dxdt
    \right|
    +\left|\int^T_0\int_{\mathbb{R}^3}  \partial_3 \psi  \partial_3^3\psi\partial_3^2 v_3\partial_3^2\psi  dxdt
    \right|
     \nonumber \\
    &\leq& C\|\nabla\psi\|_{L^\infty_T(H^2(\mathbb{R}^3))}\|\nabla_h\nabla\psi\|_{L^2_T(H^1(\mathbb{R}^3))}
            \|\nabla v\|_{L^2_T(H^2(\mathbb{R}^3))}(1+\|\nabla\psi\|_{L^\infty_T(H^2(\mathbb{R}^3))}),\label{m0-E2.7}
    \end{eqnarray}
    where   $C$ is a positive constant  independent of $T$.
\end{lem}
\begin{lem}\label{m0-L2.5}
 Under the conditions in Lemma \ref{m0-L2.3}, then there holds
    \begin{eqnarray}
     \left|\int^T_0\int_{\mathbb{R}^3}  \partial_3 v_3(\partial_3^3\psi)^2 dxdt
    \right|
    &\leq&  C B_T^3(1+B_T^2),
    \end{eqnarray}
    where   $C$ is a positive constant  independent of $T$.
\end{lem}

Accepting these two lemmas, we proceed with our proof of the key estimate
in Lemma \ref{m0-L2.3}.
By (\ref{m0-E2.19-0}) and Lemmas \ref{m0-L2.4}-\ref{m0-L2.5}, we obtain
    \begin{eqnarray}
  I_2=C\left|\int^T_0(v\cdot\nabla\psi|\Delta\psi)_{H^2}dt\right|
           &\leq& C B_T^3(1 +B_T)^2.  \label{m0-E2.20}
    \end{eqnarray}
Similarly, one can estimate that
        \begin{eqnarray}
     I_7&=&C\left|\int^T_0(\Delta(v\cdot\nabla\psi)|\Delta\psi)_{H^1}dt\right|\nonumber\\
            &=& C\left| \sum_{i=1}^3
            \int^T_0 (\partial_i(v\cdot\nabla\psi)|
            \partial_i\Delta\psi)_{H^1}dt\right| \nonumber\\
            &\leq& C B_T^3(1+ B_T)^2.
    \end{eqnarray}
Next, we estimate $I_3$ as follows:
    \begin{eqnarray}
      I_3&=&C\left|\int^T_0(\mathrm{div}(\nabla\psi\otimes\nabla\psi)|v)_{H^2}dt\right|\nonumber\\
            &=&C\left| \sum_{|\alpha|\leq2}\sum_{i,j=1}^3\int^T_0\int
            \partial^\alpha(\partial_i\psi\partial_j\psi)\partial^\alpha\partial_i v_j dxdt\right|
            \nonumber\\
      &\leq& C\|\nabla v\|_{L^2_T(L^2(\mathbb{R}^3))}\|\nabla\psi\|_{L^4_T(L^4(\mathbb{R}^3))}^2
      +C\|\nabla^2 v\|_{L^2_T(L^2(\mathbb{R}^3))}\|\nabla\psi\|_{L^4_T(L^4(\mathbb{R}^3))}\|\nabla^2\psi\|_{L^4_T(L^4(\mathbb{R}^3))}
      \nonumber\\
                &&+C\|\nabla^3v\|_{L^2_T(L^2(\mathbb{R}^3))}
                (\|\nabla^2\psi\|_{L^4_T(L^4(\mathbb{R}^3))}^2
                +\|\nabla_h\psi\|_{L^2_T(L^\infty(\mathbb{R}^3))}\|\nabla^3\psi\|_{L^\infty_T(L^2(\mathbb{R}^3))}
                \nonumber\\
      &&+\|\nabla\psi\|_{L^\infty_T(L^\infty(\mathbb{R}^3))}\|\nabla^2\nabla_h\psi\|_{L^2_T(L^2(\mathbb{R}^3))}
                )+C\left|\int^T_0\int\partial_3^3v_3\partial_3\psi\partial_3^3\psi dxdt\right|
                \nonumber\\
    &\leq& C B_T^3.\label{m0-E2.23}
    \end{eqnarray}
    Apply the same line of arguments, one can deduce that
    \begin{eqnarray}
     &&\left|\sum_{i,j=1}^3\int^T_0
      (\partial_3(-\Delta)^{-1}(\partial_iv_j\partial_j v_i)|\Delta\psi)_{H^1}dt\right|\nonumber\\
            &\leq&C\int^T_0\|(\nabla v)^2\|_{H^1(\mathbb{R}^3)}\|\nabla\psi\|_{H^1(\mathbb{R}^3)} dt\nonumber\\
      &\leq&  C\|\nabla
v\|_{L^2_T(H^2(\mathbb{R}^3))}^2\|\nabla\psi\|_{L^\infty_T(H^2(\mathbb{R}^3))}. \label{m0-E2.24}
    \end{eqnarray}

    Similarly, one has
       \begin{eqnarray}
    &&\left|\int^T_0
    \left(\left.-\sum_{i,j=1}^3\partial_3(-\Delta)^{-1}(
    \partial_i\partial_j(\partial_i\psi\partial_j\psi))-
    \sum_{j=1}^3\partial_j(\partial_3\psi\partial_j\psi)\right|\Delta\psi
    \right)_{H^1}
    dt\right|\nonumber\\
        &=&\left|\int^T_0
    \left(\left.-\sum_{i,j=1}^2\partial_3(-\Delta)^{-1}(
    \partial_i\partial_j(\partial_i\psi\partial_j\psi))  \right|\Delta\psi
    \right)_{H^1}
    dt\right|\nonumber\\
        &&+\left|\int^T_0
   \sum_{j=1}^2 \left(\left.   2\partial_3(-\Delta)^{-1} (
    \partial_3(\partial_3\psi\partial_j\psi))+
     (\partial_3\psi\partial_j\psi)\right|\partial_j\Delta\psi
    \right)_{H^1}
    dt\right|\nonumber\\
    &&+\left|\int^T_0
     ( -  (-\Delta)^{-1}
    \partial_3^3 (\partial_3\psi)^2  -
     \partial_3(\partial_3\psi )^2 |\Delta\psi
     )_{H^1}dt\right|\nonumber\\
            &\leq& C\int^T_0
            \|\nabla (\nabla_h\psi)^2\|_{H^1(\mathbb{R}^3)}\|\Delta\psi\|_{H^1(\mathbb{R}^3)}dt
            +C\int^T_0
            \|\nabla (\nabla \psi\nabla_h\psi) \|_{H^1(\mathbb{R}^3)}\|\nabla\nabla_h\psi\|_{H^1(\mathbb{R}^3)}dt
            \nonumber\\
    &&+\left|\int^T_0
     ( \sum_{i=1}^2  (-\Delta)^{-1}
    \partial_3\partial_i^2 (\partial_3\psi)^2 |\Delta\psi
     )_{H^1}dt\right|\nonumber\\
            &\leq& C\|\nabla\nabla_h \psi\|_{L^2_T(H^1(\mathbb{R}^3))}^2\|\nabla\psi\|_{L^\infty_T(H^2(\mathbb{R}^3))}
            +C\int^T_0\|\nabla_h(\nabla\psi)^2\|_{H^1(\mathbb{R}^3)}\|\nabla\nabla_h\psi\|_{H^1(\mathbb{R}^3)}dt\nonumber\\
            &\leq& C\|\nabla\nabla_h \psi\|_{L^2_T(H^1(\mathbb{R}^3))}^2\|\nabla\psi\|_{L^\infty_T(H^2(\mathbb{R}^3))}.
            \label{m0-E2.25}
        \end{eqnarray}
Combining (\ref{m0-E2.24})-(\ref{m0-E2.25}), one concludes
    \begin{equation}
      I_5=C\left| \int^T_0
       (f^v|\Delta\psi)_{H^1}dt\right|\leq CB_T^3.            \label{m0-E2.26}
    \end{equation}

One can obtain the following estimates in the same way, to save the ink,
we omit the details.
    \begin{equation}
     I_1=C\left| \int^T_0(v\cdot\nabla v|v)_{H^2}dt\right|\leq C\|\nabla v\|_{L^2_T(H^2(\mathbb{R}^3))}^2\|v\|_{L^\infty_T(H^2(\mathbb{R}^3))},
    \end{equation}
        \begin{equation}
           I_4=C\left|\int^T_0 (v\cdot\nabla v_3|\Delta\psi)_{H^1}dt\right|
           \leq C\|\nabla v\|_{L^2_T(H^2(\mathbb{R}^3))}^2\|\nabla \psi\|_{L^\infty_T(H^2(\mathbb{R}^3))},
        \end{equation}
        \begin{equation}
        I_6=C\left|  \int^T_0 (\nabla v_3|\nabla (v\cdot\nabla \psi))_{H^1}dt\right|
          \leq C\|\nabla v\|_{L^2_T(H^2(\mathbb{R}^3))}^2\|\nabla \psi\|_{L^\infty_T(H^2(\mathbb{R}^3))},\label{m0-E2.34}
            \end{equation}
Summing up  (\ref{m0-E2.18-00}), (\ref{m0-E2.20})-(\ref{m0-E2.23}) and
(\ref{m0-E2.26})-(\ref{m0-E2.34}), we conclude (\ref{m0-E2.19}).
\end{proof}

\begin{rem}
  In several places of our proofs,  we have used the fact that
    $$
    \|\nabla_h\psi\|_{L^2_T(L^\infty(\mathbb{R}^3))}\leq C    \|\nabla_h\psi\|_{L^2_T(W^{1,6}(\mathbb{R}^3))}
    \leq C   \|\nabla\nabla_h\psi\|_{L^2_T(H^{1}(\mathbb{R}^3))},
    $$
    which is a direct consequence of Sobolev embedding Theorem.
When the spatial dimension is two,  we cannot use
$\|\nabla\nabla_h\psi\|_{L^2_T(H^{1}(\mathbb{R}^2))}$ to bound
$\|\nabla_h\psi\|_{L^2_T(L^\infty(\mathbb{R}^2))}$. So it is necessary
to make various changes in order for the proofs in this article to work in
the case that the spatial dimension is two.  On the other hand, if one
assume that the initial data are in
$H^2\times \dot{H}^{-s}(\mathbb{R}^2)$, $s\in (\frac{1}{2},1)$, then use $\|\nabla^{1+s}\nabla_h\psi\|_{L^2_T(H^{1}(\mathbb{R}^2))}$ to bound
$\|\nabla_h\psi\|_{L^2_T(L^\infty(\mathbb{R}^2))}$ and obtain the similar
result though the arguments are technically more complicated.
\end{rem}

\section{Proof of Theorem \ref{m0-Thm1.2}}\label{m0-S3}
Via the analysis in \cite{Majda84}, one can get the following local
existence result following now the standard arguments:
\begin{thm}\label{m0-Thm3.1}
Assume that the initial data $(\psi_0,v_0)$ satisfy $(\nabla \psi_0,v_0)\in H^2(\mathbb{R}^3)\times H^2(\mathbb{R}^3)$,   then there exists
     $T_0>0$ such that
     the system (\ref{m0-E1.5-N3}) has a unique local solution $(\psi,v,\nabla p)$  on $[0,T_0]$ satisfying
        \begin{equation}
          \nabla \psi, v\in C([0,T_0];H^2(\mathbb{R}^3)),\  \nabla v\in L^2([0,T_0];H^2(\mathbb{R}^3)), \label{m0-E3.1}
        \end{equation}
        \begin{equation}
          \nabla p\in L^\infty([0,T_0];H^{ 1} (\mathbb{R}^3)) .\label{m0-E3.3}
        \end{equation}
\end{thm}
\noindent\textbf{Proof of Theorem \ref{m0-Thm1.2}.}
Theorem \ref{m0-Thm3.1} implies that the system (\ref{m0-E1.5-N3}) has a unique
 local strong solution $(\psi,v,\nabla p)$ on $[0,T^*)$, where $[0,T^{*})$
is the maximal existence time interval for the above solution. Our goal
is to prove $T^*=\infty$ provided that the initial data $(\psi_0,v_0)$ satisfy (\ref{1.6}).

Assume that $(\psi,v,\nabla p)$ is the unique local strong solution of (\ref{m0-E1.5-N3}) on $[0,T^*)$, and satisfies (\ref{m0-E3.1})-(\ref{m0-E3.3}).
From  (\ref{m0-E2.19}), we have
    \begin{equation}
    B_T^2    \leq C
               (\|v_0\|_{H^2(\mathbb{R}^3)}^2+\|\nabla\psi_0\|_{H^2(\mathbb{R}^3)}^2)
               + C  B_T^3(1+ B_T)^2,
      \end{equation}
for all $T\in (0,T^*)$.
If  the initial data $(\psi_0,v_0)$ satisfy (\ref{1.6}),  where $c_0$ satisfies
\begin{equation}
   C\sqrt{2C }c_0(1+ \sqrt{2C }c_0)^2\leq \frac{1}{2},
\end{equation}
then one can easily obtain
    \begin{equation}
     B_T^2\leq 2CB_0^2, \ \textrm{ for all }T\in (0,T^*).
    \end{equation}
As the right hand side of the last inequality above remains to be small,
we must have that $T^*=\infty$,  and hence (\ref{1.7}) holds.
From (\ref{m0-E1.6}), we see that $\nabla p\in L^\infty([0,\infty);H^1)$
and (\ref{1.8}) holds. This finishes the proof of  Theorem
\ref{m0-Thm1.2}.
{\hfill
$\square$\medskip}

\section*{Appendix}
Here we shall give the proofs of two technical Lemmas
\ref{m0-L2.4}-\ref{m0-L2.5} that are needed in establishing the key \textit{a
priori} estimates. As $\mathrm{div}v=0$, we can replace
$\mathrm{div}_hv_h=\partial_1 v_1+\partial_2 v_2$  by $\partial_3 v_3$
in various calculations in the proof of Lemma \ref{m0-L2.4}. Sometime, it would
be useful (and it may be also necessary) to replace $v_3$. In fact,
via (\ref{m0-E1.5-N3})$_1$, we can re-write
    \begin{equation}
      v_3=-(\partial_t\psi+v\cdot\nabla \psi ).\label{m0-E2.5}
    \end{equation}
The above substitution for $v_3$  has the advantage that it reduces
space-time integral estimates to estimates on time slices and space times
integral with higher order nonlinearities and fast dissipation. The latter
is smaller by the initial smallness assumptions.

\noindent \textbf{Proof of Lemma \ref{m0-L2.4}.}
  Using the integration by parts, the fact that $\mathrm{div} v=0$, the
 H\"{o}lder's inequality and the Sobolev embedding Theorem,
  we can estimate the first term in the lemma \ref{m0-L2.4} as follows:
    \begin{eqnarray*}
      &&\left|\int^T_0\int  \partial_3\psi  \partial_3^3\psi\partial_3^3 v_3  dxdt
    \right|\\
            &=&\left|\int^T_0\int  \partial_3\psi  \partial_3^3\psi\partial_3^2\mathrm{div}_h v_h  dxdt
    \right|\\
    &=&\left| \int^T_0\int \left(-\partial_3^2  v_h\cdot\nabla_h\partial_3\psi  \partial_3^3\psi
    -\partial_3\psi \partial_3^2  v_h\cdot\nabla_h\partial_3^3\psi
    \right)  dxdt
    \right|\\
    &=&\left| \int^T_0\int \left(-\partial_3^2  v_h\cdot\nabla_h\partial_3\psi  \partial_3^3\psi
    +\partial_3^2\psi \partial_3^2  v_h\cdot\nabla_h\partial_3^2\psi
    +\partial_3\psi \partial_3^3  v_h\cdot\nabla_h\partial_3^2\psi
    \right)  dxdt
    \right|\\
    &\leq& C\|\partial_3^2  v_h\|_{L^2_T(L^4(\mathbb{R}^3))}\|\nabla_h\partial_3 \psi\|_{L^2_T(L^4(\mathbb{R}^3))}
    \|\partial_3^3\psi\|_{L^\infty_T(L^2(\mathbb{R}^3))}\\
    &&+C\|\partial_3^2\psi\|_{L^\infty_T(L^4(\mathbb{R}^3))}
    \|\partial_3^2 v_h\|_{L^2_T(L^4(\mathbb{R}^3))}\|\nabla_h\partial_3^2\psi\|_{L^2_T(L^2(\mathbb{R}^3))}\\
            &&
    +C\|\partial_3\psi\|_{L^\infty_T(L^\infty(\mathbb{R}^3))}
    \|\partial_3^3 v_h\|_{L^2_T(L^2(\mathbb{R}^3))}\|\nabla_h\partial_3^2\psi\|_{L^2_T(L^2(\mathbb{R}^3))}\\
            &\leq& C\|\nabla\psi\|_{L^\infty_T(H^2(\mathbb{R}^3))}\|\nabla_h\nabla\psi\|_{L^2_T(H^1(\mathbb{R}^3))}
            \|\nabla v\|_{L^2_T(H^2(\mathbb{R}^3))}.
    \end{eqnarray*}
The other terms in (\ref{m0-E2.7}) can be treated similarly, and we can
conclude Lemma \ref{m0-L2.4}.
  {\hfill
$\square$\medskip}

We shall now proceed with the proof lemma \ref{m0-L2.5}. The basic
strategy has been described earlier, see for example, (\ref{idea}).
For this purpose, we first prove the following lemma. Here we use the
equation (\ref{m0-E1.5-N3})$_1$  to bounded the term $\int^T_0\int \partial_3\psi \partial_3 v_3(\partial_3^3\psi)^2 dxdt$.
\begin{lem}\label{m0-L2.2-2}
  Under the conditions of Lemma \ref{m0-L2.3}, then there holds
    \begin{eqnarray}
    \left|\int^T_0\int  \partial_3\psi \partial_3 v_3(\partial_3^3\psi)^2 dxdt
    \right|
    &\leq& CB_T^4 (1+B_T), \label{m0-E2.5-0}
    \end{eqnarray}
    where   $C$ is a positive constant  independent of $T$.
\end{lem}
\begin{proof}
Applying (\ref{m0-E2.2}), (\ref{m0-E2.5}), the integration by parts,
H\"{o}lder's inequality and Sobolev embedding Theorem, we get
    \begin{eqnarray*}
      &&\left|\int^T_0\int  \partial_3\psi \partial_3 v_3(\partial_3^3\psi)^2 dxdt
    \right|\\
    &=&\left|\int^T_0\int   \partial_3\psi \partial_3
    (\partial_t\psi+v\cdot\nabla \psi )(\partial_3^3\psi)^2 dxdt
    \right|\\
    &=&\left| \left.\int \frac{1}{2 } (\partial_3\psi)^2 (\partial_3^3\psi)^2 dx\right|^T_0
    +\int^T_0\int \left[ - (\partial_3\psi)^2 \partial_3^3\psi\partial_3^3\partial_t\psi
    + \partial_3\psi \partial_3
    (v\cdot\nabla \psi )(\partial_3^3\psi)^2 \right]dxdt
     \right|\\
        &\leq& C\|\partial_3\psi\|_{L^\infty_T(L^\infty(\mathbb{R}^3))}^2
        \|\partial_3^3\psi\|_{L^\infty_T(L^2(\mathbb{R}^3))}^2 \\
        && +\left|
    \int^T_0\int \left[  (\partial_3\psi)^2 \partial_3^3\psi\partial_3^3(v\cdot\nabla\psi + v_3)
    +\partial_3\psi\partial_3(v\cdot\nabla\psi)   (\partial_3^3\psi)^2 \right]dxdt
     \right|\\
      &\leq& C\|\nabla\psi\|^4_{L^\infty_T(H^2(\mathbb{R}^3))}   +C\|\partial_3\psi\|_{L^4_T(L^\infty(\mathbb{R}^3))}^2
      \|\partial_3^3\psi\|_{L^\infty_T(L^2(\mathbb{R}^3))}
        \|\partial_3^3(v\cdot\nabla\psi)-v\cdot\nabla\partial_3^3\psi\|_{L^2_T(L^2(\mathbb{R}^3))} \\
        && +C\|\partial_3\psi\|_{L^4_T(L^\infty(\mathbb{R}^3))}^2
      \|\partial_3^3\psi\|_{L^\infty_T(L^2(\mathbb{R}^3))}
        \|\partial_3^3v_3        \|_{L^2_T(L^2(\mathbb{R}^3))}\\
        &&+\left|
    \int^T_0\int \left\{\frac{1}{2 }
     \left[(\partial_3\psi)^2  v\cdot\nabla(\partial_3^3\psi)^2
   +
     v\cdot\nabla (\partial_3\psi)^2 (\partial_3^3\psi)^2\right] \right\}dxdt
     \right|\\
     &&+\left|
    \int^T_0\int \left\{ \partial_3\psi\partial_3 v\cdot\nabla \psi  (\partial_3^3\psi)^2 \right\}dxdt
     \right|\\
              &\leq& CB_T^4(1+B_T).
    \end{eqnarray*}
In deriving the last inequality above, we have used the following
calculations:
    \begin{eqnarray*}
      &&\|\partial_3^3(v\cdot\nabla\psi)-v\cdot\nabla\partial_3^3\psi\|_{L^2_T(L^2(\mathbb{R}^3))}\\
        &\leq&C  \|\nabla^3 v\|_{L^2_T(L^2(\mathbb{R}^3))}\|\nabla \psi\|_{L^\infty_T(L^\infty(\mathbb{R}^3))}
        +C\|\nabla^2 v\|_{L^2_T(L^4(\mathbb{R}^3))}\|\nabla^2 \psi\|_{L^\infty_T(L^4(\mathbb{R}^3))}
       \\
        && +C\|\nabla v\|_{L^2_T(L^\infty(\mathbb{R}^3))}\|\nabla^3 \psi\|_{L^\infty_T(L^2(\mathbb{R}^3))}\\
        &\leq&C\|\nabla v\|_{L^2_T(H^2(\mathbb{R}^3))}\|\nabla  \psi\|_{L^\infty_T(H^2(\mathbb{R}^3))},
    \end{eqnarray*}
and also the estimation:
    \begin{eqnarray*}
      &&\left|
    \int^T_0\int \left\{ \partial_3\psi\partial_3 v\cdot\nabla \psi  (\partial_3^3\psi)^2 \right\}dxdt
     \right|     \\
     &\leq& \|\nabla\psi\|_{L^4_T(L^\infty(\mathbb{R}^3))}^2\|\nabla v\|_{L^2_T(L^\infty(\mathbb{R}^3))}\|\nabla^3\psi\|_{L^\infty_T(L^2(\mathbb{R}^3))}^2.
    \end{eqnarray*}
\end{proof}

\noindent\textbf{Proof of Lemma \ref{m0-L2.5}.}
    We use Lemma \ref{m0-L2.2}, (\ref{m0-E2.7}), (\ref{m0-E2.5}), (\ref{m0-E2.5-0}),
 the integration by parts,  H\"{o}lder's inequality and Sobolev embedding Theorem
to do following derivations:
    \begin{eqnarray*}
      &&\left|\int^T_0\int  \partial_3 v_3(\partial_3^3\psi)^2 dxdt
    \right|\\
        &=&\left|\int^T_0\int  \partial_3 (\partial_t
        \psi+v\cdot\nabla\psi )(\partial_3^3\psi)^2 dxdt
    \right|\\
        &=&\left|\left.\int  \partial_3
        \psi (\partial_3^3\psi)^2 dx\right|^T_0
        +\int^T_0\int \left\{  -2\partial_3\psi\partial_3^3\psi\partial_3^3\partial_t\psi
         + \partial_3 ( v\cdot\nabla\psi )(\partial_3^3\psi)^2\right\}dxdt
    \right|\\
        &\leq& C\|\partial_3\psi\|_{L^\infty_T(L^\infty(\mathbb{R}^3))}\|\partial_3^3\psi\|_{L^\infty_T(L^2(\mathbb{R}^3))}^2 \\
            &&  +\left|
        \int^T_0\int \left\{  2\partial_3\psi\partial_3^3\psi\partial_3^3(v\cdot\nabla\psi+
        v_3) + \partial_3 ( v\cdot\nabla\psi )(\partial_3^3\psi)^2\right\}dxdt
    \right|\\
        &\leq& C\|\nabla\psi\|_{L^\infty_T(H^2(\mathbb{R}^3))}^3
        +C\|\nabla \psi\|_{L^\infty_T(H^2(\mathbb{R}^3))}\|\nabla\nabla_h\psi\|_{L^2_T(H^1(\mathbb{R}^3))}\|\nabla v\|_{L^2_T(H^2(\mathbb{R}^3))}\\
           &&  +\left|
        \int^T_0\int   \left\{ 2\partial_3\psi\partial_3^3\psi[\partial_3^3(v\cdot\nabla\psi)-v\cdot\nabla\partial_3^3\psi]
         +
           \partial_3\psi  v \cdot\nabla(\partial_3^3\psi)^2    \right.\right.\\
        &&\left.\left. + v\cdot\nabla \partial_3\psi(\partial_3^3\psi)^2 + \partial_3   v\cdot\nabla\psi (\partial_3^3\psi)^2\right\}dxdt
    \right|\\
    &\leq&
    C B_T^3(1+B_T^2).
    \end{eqnarray*}
    In the last step above, we have also applied the following estimate,
    \begin{eqnarray*}
    &&   \left|
        \int^T_0\int   \left\{ 2\partial_3\psi\partial_3^3\psi[\partial_3^3(v\cdot\nabla\psi)-v\cdot\nabla\partial_3^3\psi]
          + \partial_3   v\cdot\nabla\psi (\partial_3^3\psi)^2\right\}dxdt
    \right|\\
        &\leq& \left|
        \int^T_0\int    2\partial_3\psi\partial_3^3\psi(\partial_3^3v\cdot\nabla\psi+3\partial_3^2v\cdot\nabla\partial_3\psi
                 )  dxdt
    \right|+\left|
        \int^T_0\int   6\partial_3\psi\partial_3^3\psi
        \partial_3 v_h\cdot\nabla_h\partial_3^2\psi   dxdt
    \right|\\
    &&
         +\left|
        \int^T_0\int  7\partial_3\psi
        \partial_3 v_3 (\partial_3^3\psi)^2   dxdt
    \right|  +\left|
        \int^T_0\int
        \partial_3 v_h\cdot\nabla_h \psi (\partial_3^3\psi)^2   dxdt
    \right|\\
        &\leq&C\|\nabla\psi\|_{L^4_T(L^{ \infty}(\mathbb{R}^3))}^2\|\partial_3^3\psi\|_{L^\infty_T(L^2(\mathbb{R}^3))}\|\nabla^3 v\|_{L^2_T(L^2(\mathbb{R}^3))}
        \\
        &&+C\|\nabla\psi\|_{L^4_T(L^{ \infty}(\mathbb{R}^3))}\|\partial_3^3\psi\|_{L^\infty_T(L^2(\mathbb{R}^3))}\|\nabla^2 v\|_{L^2_T(L^4(\mathbb{R}^3))}
        \|\nabla^2\psi\|_{L^4_T(L^4(\mathbb{R}^3))}
      \\
         &&+C\|\partial_3\psi\|_{L^\infty_T(L^\infty(\mathbb{R}^3))}
            \|\partial_3^3\psi\|_{L^\infty_T(L^2(\mathbb{R}^3))}\|\partial_3 v\|_{L^2_T(L^\infty(\mathbb{R}^3))}
            \|\nabla_h\partial_3^2\psi\|_{L^2_T(L^2(\mathbb{R}^3))}\\
     &&+    CB_T^4(1+B_T) +C\|\partial_3 v\|_{L^2_T(L^\infty(\mathbb{R}^3))}
     \|\nabla_h\psi\|_{L^2_T(L^\infty(\mathbb{R}^3))}\|\partial_3^3\psi\|_{L^\infty_T(L^2(\mathbb{R}^3))}^2  \\
      &\leq&  C B_T^4(1+B_T).
    \end{eqnarray*}

The remaining parts have already shown to have the desired estimates. Thus
we complete the proof of Lemma.
{\hfill
$\square$\medskip}

\section*{Acknowledgements}
The research of F.H.Lin is partial supported by the NSF grants, DMS 1065964
and DMS 1159313. The research of T. Zhang is partially supported by NSF of
China under Grants 11271322,  11331005 and 11271017, National Program for
Special Support of Top-Notch Young Professionals, Program for New Century
Excellent Talents in University NCET-11-0462, the Fundamental Research
Funds for the Central Universities (2012QNA3001). Part of the work was
done while the second author was visiting the Courant Institute
Mathematical Sciences. T.Z. wants to thank the Courant Institute for the
warm hospitality.

\end{document}